\documentclass[12pt]{amsart}
\usepackage{graphicx} 
\usepackage[colorlinks,linkcolor=blue,citecolor=blue,urlcolor=blue, pdfcenterwindow, pdfstartview={XYZ null null 1.2}, pdffitwindow, pdfdisplaydoctitle=true]{hyperref}
\usepackage{array}
\usepackage{threeparttable}
\usepackage{rotating}
\usepackage[section]{placeins}
\usepackage{amssymb} 
\usepackage{amsmath}
\usepackage{mathrsfs}
\usepackage{mathabx}
\usepackage{envmath}
\usepackage{enumerate}
\usepackage{color}
\usepackage{relsize}
\usepackage[T1]{fontenc}
\usepackage[all]{xy}
\usepackage{comment}

\begin{document}
\bibliographystyle{alpha}
\newtheorem{lemma}{Lemma}[section]
\newtheorem{lemm}[lemma]{Lemma}
\newtheorem{prop}[lemma]{Proposition}
\newtheorem{coro}[lemma]{Corollary}
\newtheorem{theo}[lemma]{Theorem}
\newtheorem{conj}[lemma]{Conjecture}
\newtheorem{prob}{Problem}
\newtheorem{ques}{Question}

\newtheorem{rema}[lemma]{Remark}
\newtheorem{rems}[lemma]{Remarks}
\newtheorem{defi}[lemma]{Definition}
\newtheorem{defis}[lemma]{Definitions}
\newtheorem{exam}[lemma]{Example}

\newcommand{\N}{\mathbf N}
\newcommand{\Z}{\mathbf Z}
\newcommand{\R}{\mathbf R}
\newcommand{\Q}{\mathbf Q}
\newcommand{\C}{\mathbf C}

\title{ Schur's orthogonality relations for the free group}
\begin{abstract}
We prove an explicit convergence of suitably normalized integrals on balls where the integrand is the product of coefficients of the quasi-regular representation of the finitely generated free group. This result follows from the fact that the quasi-regular representation of the free group is c-tempered in the sens of Kazhdan and Yom Din \cite[ def 2.1]{kazhdan2022tempered}.
The convergence can be summed up by the formula:
\begin{equation}\nonumber
     \underset{n \to \infty}{\lim}\frac{1}{n^{3}}\int_{B_n}\langle\pi(g)\psi_1,\psi_2\rangle\overline{\langle\pi(g)\psi_3,\psi_4\rangle}d\mu(g) = \frac{3q(q+1)}{(q-1)^2}\langle\psi_1,\psi_3\rangle\overline{\langle\psi_2,\psi_4\rangle}  
\end{equation}
where $q+1$ is twice the rank of the free group, $B_n$ is the ball of radius $n$ with respect to the canonical word metric and $\psi_1,\psi_2,\psi_3,\psi_4$ are square summable functions on the boundary of the free group.
\end{abstract}


\author{Guillaume Delord}
\email{guillaume.delord.1@etu.univ-amu.fr}
\date{November 2024}
\thanks{Thanks to Ch. Pittet, A. Boyer and A. Pinochet Lobos for their great help and generosity}

\maketitle

\tableofcontents

\section{Introduction.}\label{Introduction}
Given an irreducible unitary representation of a compact group $\rho:G\to\mathcal{U}(\mathcal{V})$, we know (see for instance \cite{bendikov2023braidingasymptoticschursorthogonality}) that the Hilbert space $\mathcal{V}$ is finite dimensional (say $d:=dim_{\C}(\mathcal{V})$) and, for all $v_1,v_2,v_3$ and $v_4$ in $\mathcal{V}$, we have:
\begin{equation}\nonumber
\int_{G}\langle\rho(g)v_1,v_2\rangle\overline{\langle\rho(g)v_3,v_4\rangle}dg=\frac{1}{d}\langle v_1,v_3\rangle\overline{\langle v_2,v_4\rangle}
\end{equation}
where $dg$ denotes the normalised Haar measure on $G$.\\
This formula (Schur's orthogonality relations) can be seen as a generalisation of the fact that characters of finite groups are unitary in $l^2(G,\C)$ but does not make sense when $G$ is not compact and the coefficients not square summable.\\
In this paper, we compute an equivalent asymptotic formula for the boundary representation of the free group.\\
In \cite[First theorem]{boyer2016asymptotic}, one can find similar results for Gromov hyperbolic groups when the metric is non arithmetic .
\subsection{Settings and notations.}
Let $G$ be the free group with $N$ generators, $X$ the Cayley graph associated to right multiplication in $G$ and $x_0$ a base point in $X$. Then $X$ is an homogeneous tree of degree $2N=:q+1$ equipped with the unique distance $d$ which gives the value 1 to any pair of adjacent vertices (see figure \ref{fig:CayleyGraph} for an example where $N=2$ and $x_0=e$ the neutral element of $G$). We denote by $[x,y]$ the unique geodesic joining $x$ to $y$ in $X$, $S_k:=S(x_0,k):=\{x\in X||x|:=d(x_0,x)=k\}$ the sphere centered at $x_0$ with radius $k \in \N$ (for an element $g$ of $G$, we write also $g \in S_k$ whenever $gx_0 \in S(x_0,k)$ as an element of $X$). We also denote by $B_n:=B(x_0,n)$ the ball with radius $n \in \N$.\\
\begin{figure}
    \centering
    \includegraphics[width=0.5\linewidth]{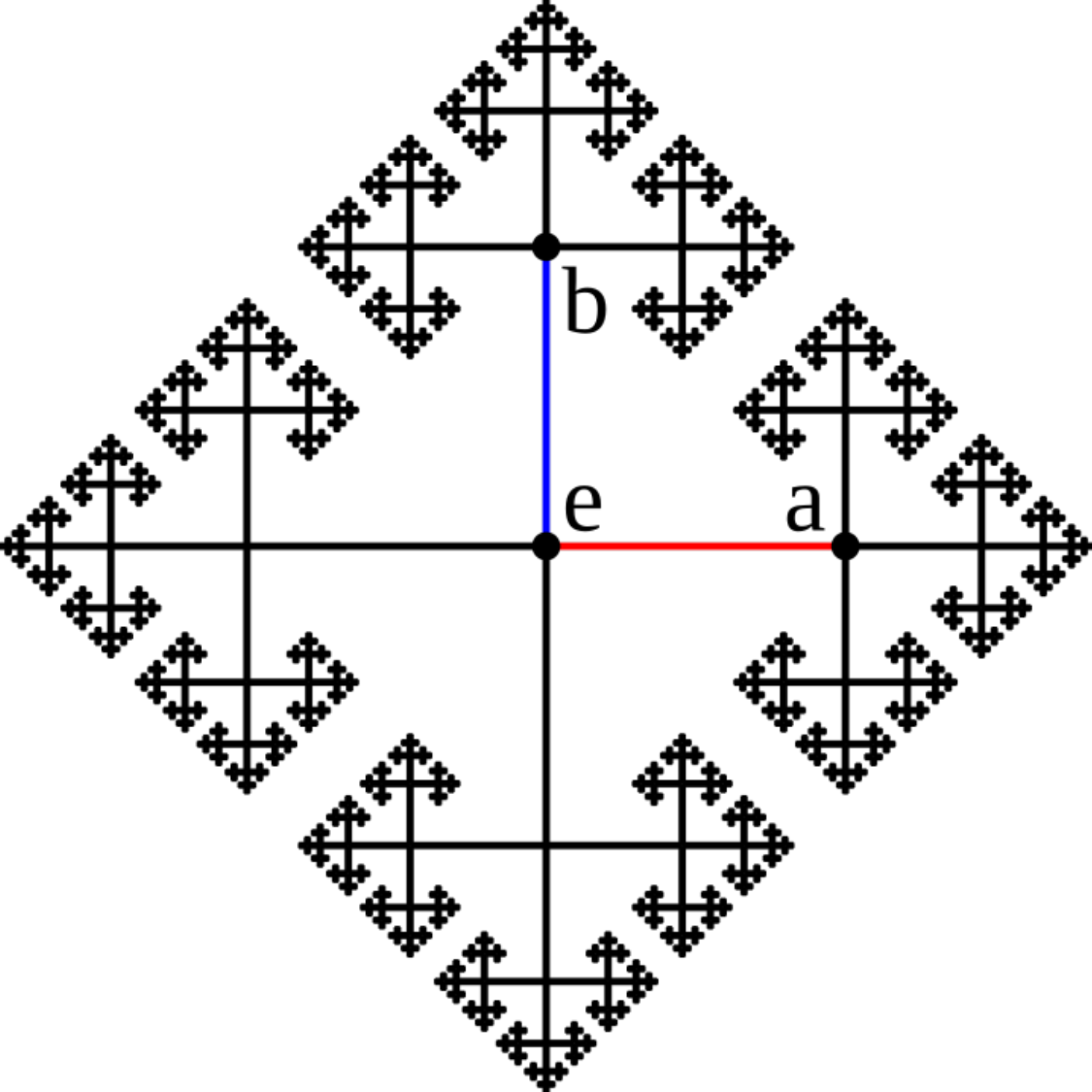}
    \caption{Example for $N=2$ and $x_0=e$}
    \label{fig:CayleyGraph}
\end{figure}
A point $\omega$ of the boundary $\Omega$ can be seen as a direction to infinity or, more precisely, as an infinite geodesics $[x_0,\omega)$ starting at $x_0$. We equip $\Omega$ with a topological structure declaring its basis of open sets to be all the shadows $\Omega_x=\{\omega \in \Omega| [x_0,x] \subset [x_0,\omega)\}$, where $x$ is in $X$, which makes $\Omega$ a compact topological space (for more details, see the introduction of \cite{kuhn1992boundary} where Kuhn embeds $X \bigsqcup \Omega$ in  a  cartesian  product  of  compact spaces).\\
We also equip $\Omega$ with a Borel probability measure $\nu$ which satisfy: 
\begin{equation}\nonumber
    \forall x \in X, \nu(\Omega_x) = \frac{1}{|S_{|x|}|}=\frac{1}{(q+1)q^{|x|-1}}.
\end{equation}
The isometric left action of $G$ on $X$ clearly extends to a left action on $\Omega$ and one can show that $\nu$ is quasi-invariant under this action (see \cite[ Corollary 2.6.3]{1573387449062997504} where Bourdon shows that $G$ acts by conformal maps on $\Omega$).\\
In particular, for all $g$ in $G$, $g_*\nu << \nu$. One can show that $\frac{dg_*\nu}{d\nu}(\omega)=P(g^{-1},\omega):=q^{\beta_{\omega}(x_0,gx_0)}$ where $\beta$ is the Busemann function (see also \cite{1573387449062997504}).\\
Let $\mathcal{H}$ be the Hilbert space $L^2(\Omega,\nu,\C)$. We define the unitary representation $\pi:G \to \mathcal{U}(\mathcal{H})$ by:
\begin{equation}\nonumber
    \forall g \in G, \psi \in \mathcal{H}, \omega \in \Omega, \pi(g)\psi(\omega) := P(g^{-1},\omega)^{\frac{1}{2}}\psi(g^{-1}\omega).
\end{equation}
It is well known $\pi$ is irreducible (see for example \cite[thm 1.2]{boyer2017ergodic}). In particular, by Schur's lemma:
\begin{equation}\label{intertwiners of pi}
    Hom_G(\pi,\pi)=\{T \in \mathcal{B}(\mathcal{H}) \text{ such that } T \circ \pi(g) =\pi(g) \circ T, \forall g \in G\}=\C Id_{\mathcal{H}}.
\end{equation}

\subsection{Organisation of the paper.}
In section \ref{computing H Ch} we compute the values of the Harish-Chandra function and prove that it is spherical using partitions of the boundary where the Busemann function is constant. In section \ref{pi c-temepered}, we show (using results from \cite{boyer2017ergodic}) that the representation is c-tempered in the sense of \cite{kazhdan2022tempered} (using results from \cite{kuhn1994amenable} and \cite{haagerup1978example} about the extension of our representation and the regular one). To conclude, in section \ref{Asymptotic Schur's orthogonality relations}, we adapt and detail the proof one can find in \cite{kazhdan2022tempered} to obtain our asymptotic orthogonality relations by showing some intermediate results based on functional analysis from \cite{reed1981functional} and explain how we obtain our initially mentioned formula.

\section{Computing the Harish-Chandra function.}\label{computing H Ch}
\subsection{Partitions where the Busemann function is constant.}
\begin{lemma} \label{partition of finite geodesic}
    $\textit{Let } x \in S_n \textit{ such that } [x_0,x] = (x_0,x_1,\dots, x=x_n)$.
    One defines the following sets:
    \begin{equation} \nonumber
        E_k(x):=\begin{cases}
            \Omega_x  &\textit{ if } k=n\\
            \Omega_{x_k} -\Omega_{x_{k+1}}  &\textit{ if } 0 \leq k<n 
        \end{cases}
    \end{equation} 
Then $\{E_0(x),\dots, E_n(x)\}$ is a partition of $\Omega$ (see Figure \ref{Dessin2} for an example where $n=2$).\\
Moreover, $\omega \mapsto\beta_{\omega}(x_0,x)$ is constantly equal to $2k - n$ on $E_k(x)$.
\end{lemma}
\begin{figure}[t!] 
    \centering
    \includegraphics[width=0.8\linewidth]{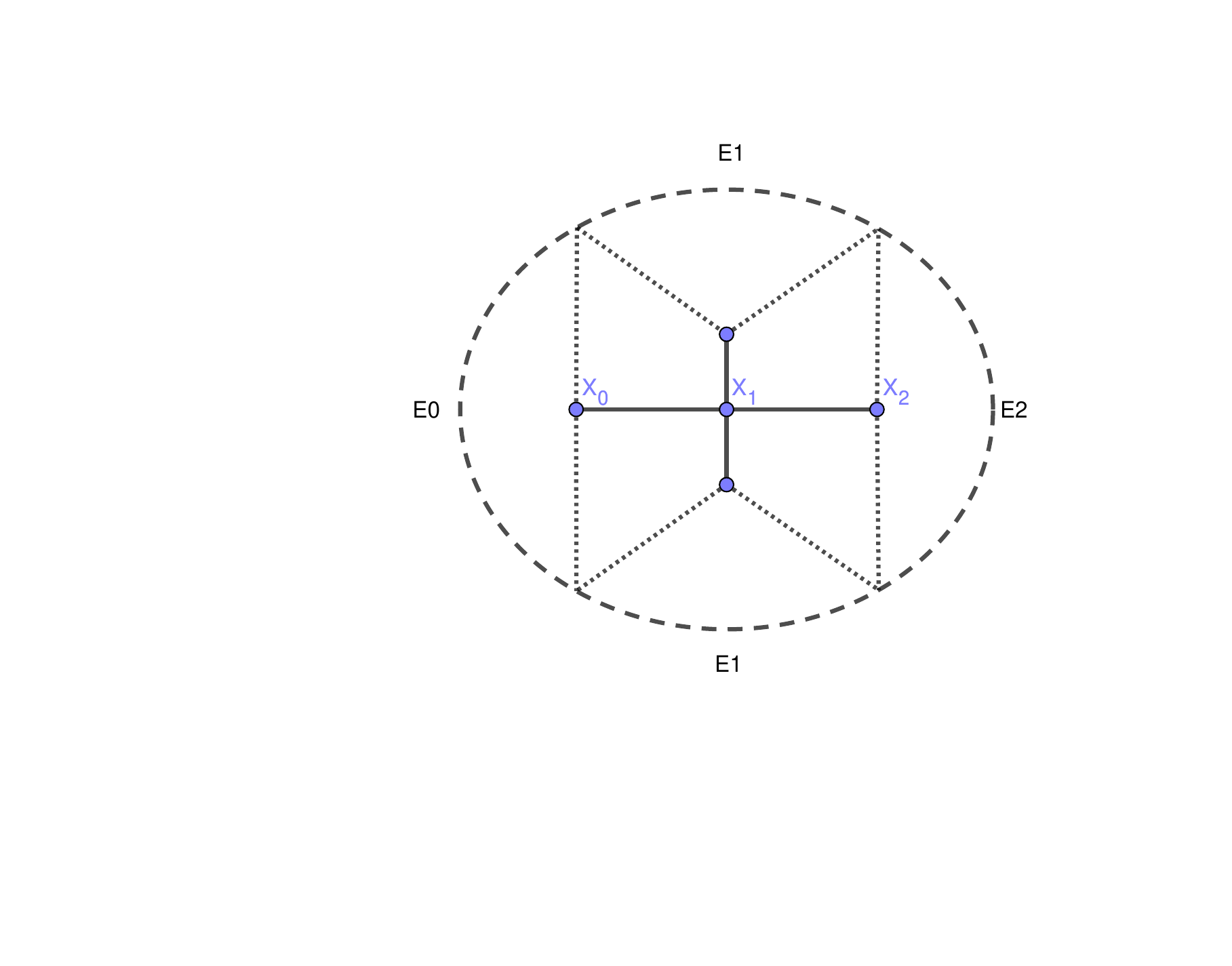}
    \caption{Example with $[x_0,x]=(x_0,x_1,x_2)$ \\
     illustrating Lemma \ref{partition of finite geodesic} where the family $(E_0,E_1,E_2)$ is a partition of the boundary.}
    \label{Dessin2}
\end{figure}

One computes:
\begin{equation}\label{E_k(x) measure}
    \nu(E_k(x))=\begin{cases}
        \nu(\Omega_x)=\frac{1}{(q+1)q^{n-1}}  &($if $ k=n)\\
    \nu(\Omega_{x_k})-\nu(\Omega_{x_{k+1}})=\frac{q-1}{(q+1)q^k} $ $&($if $ k \in [1,n-1])\\
    1 -\nu(\Omega_{x_1})=\frac{q}{q+1}  &($if $ k=0)
    \end{cases}
\end{equation}

\subsection{Computation: The Harish-Chandra function is spherical.}
Here we define $\Xi$ and show that it is constant on $S_n$ by computing its value.
\begin{equation} \nonumber
    \begin{array}{cccc}
        \Xi: & G & \to & \R^{+*}\\
         & g & \mapsto & \langle\pi(g)1_{\Omega},1_{\Omega}\rangle
    \end{array}
\end{equation}
Note that $\Xi(0):=\Xi(e)=1$\newline
Now, let $n \geq 1$, and $g$ such that $g \in S_n$ and compute \newline
\begin{align}\nonumber
    \Xi(g) & = \int_{ \Omega}\pi(g)1_{\Omega}\overline{1_{\Omega}}d\nu = \int_{ \Omega}P(g^{-1},\omega)^{\frac{1}{2}}d\nu(\omega) = \sum_{k=0}^n \int_{E_k(gx_0)}q^{\frac{\beta_{\omega}(x_0,gx_0)}{2}}d\nu(\omega)\\ \nonumber
    & = \sum_{k=0}^nq^{\frac{2k-n}{2}}\nu(E_k(gx_0)) \qquad ( \text{by Lemma } \ref{partition of finite geodesic})\\ \nonumber
    & = q^{-\frac{n}{2}}\frac{q}{q+1} + \sum_{k=1}^{n-1}q^{\frac{2k-n}{2}}\frac{q-1}{(q+1)q^k}+q^{\frac{n}{2}}\frac{1}{(q+1)q^{n-1}} \quad ( \text{from Equation } \ref{E_k(x) measure})\\ \nonumber
    & = \bigg[1+\big(\frac{q-1}{q+1}\big)n\bigg]q^{-\frac{n}{2}}.
\end{align}

So, $\Xi$ is spherical and we can define:
\begin{equation} \nonumber
    \begin{array}{ccccc}
        \Xi: & \N & \to & \R^{+*} &\\
         & n & \mapsto & \Xi(g) & ( \text{where } g \text{ is any element of  } S_n)
    \end{array}
\end{equation}

\section{$\pi$ is c-tempered.}\label{pi c-temepered}
Here we will prove that $\pi$ is c-tempered, in the sense of Kazhdan and Yom Din in section 2 \guillemotleft Notion of c-temperedness\guillemotright \cite[def 2.1]{kazhdan2022tempered}. \newline
For all subset $L \subset G$ and all $\psi_1, \psi_2 \in \mathcal{H}$, one can define the quantity:
\begin{equation}\nonumber
    M_{\psi_1,\psi_2}(L):=\int_{L}|\langle\pi(g)\psi_1,\psi_2\rangle|^2d\mu(g)
\end{equation}
In particular, we have
\begin{equation}\nonumber
    M_{1_{\Omega},1_{\Omega}}(L)= \int_{L}\Xi(g)^2d\mu(g) 
\end{equation}
Using the spherical property of $\Xi$, one computes:
\begin{equation}\label{average Chandra on sphere}
    M_{1_{\Omega},1_{\Omega}}(S_k)=\frac{q+1}{q}\left(1+\frac{q-1}{q+1}k\right)^2 \quad (k>0).
\end{equation}
and
\begin{align} \label{average Chandra on ball}
    M_{1_{\Omega},1_{\Omega}}(B_n) & = \sum_{k=0}^n M_{1_{\Omega},1_{\Omega}}(S_k)\\ \nonumber
    &= 1 + \frac{q+1}{q}\bigg[\frac{n(n+1)(2n+1)}{6}\left(\frac{q-1}{q+1}\right)^2 + n(n+1)\frac{q-1}{q+1} + n\bigg] \\ \nonumber
    &\underset{n \to \infty}{\sim} \frac{n^3}{K} \text{ where } K= \frac{3q(q+1)}{(q-1)^2}.
\end{align}
Considering the sequence $\{B_n\}_{n \in \N}$ of balls in $G$ and our unit vector $1_{\Omega} \in \mathcal{H}$. Then if the two following conditions are satisfied:
\begin{equation}\label{cTemp 1}
    \forall\psi_1,\psi_2 \in \mathcal{H}, \quad \underset{n \to \infty}{\limsup }\frac{M_{\psi_1,\psi_2}(B_n)}{M_{1_{\Omega},1_{\Omega}}(B_n)}<\infty;
\end{equation}
\begin{equation}\label{cTemp 2}
    \forall g,h \in G \text{ and }\psi_1,\psi_2 \in \mathcal{H}, \quad \underset{n \to \infty}{\lim}\frac{M_{\psi_1,\psi_2}(B_n\Delta h^{-1}B_ng)}{M_{1_{\Omega},1_{\Omega}}(B_n)} = 0.
\end{equation}
One says that $\pi$ is c-tempered with F{\o}lner sequence $\{B_n\}_{n \in \N}$ (see \cite[def 2.1]{kazhdan2022tempered}).
\begin{rema}
   The condition (\ref{cTemp 2}) is equivalent to the second condition in the definition \cite[def 2.1]{kazhdan2022tempered} because $G$ is discret, therefore its compacts are the finite sets.
\end{rema}

\begin{lemma}\label{Lemma c tempered 1}
    $\pi$ satisfies the first condition ($\ref{cTemp 1}$).
\end{lemma}
\begin{proof}
By \cite{kuhn1994amenable}, we know that $\pi$ is weakly contained in the regular representation $\pi_{reg}$.
In particular, the extensions on $l^1(G)$ of these representations:
\begin{equation}\nonumber
    \rho^{ext}(\sum_{g \in G}a_g\delta_g)=\int_Ga_g\rho(g) dg=\sum_{g \in G}a_g\rho(g) 
\end{equation} 
(where $\rho \in \{\pi, \pi_{reg}\}$) satisfy:
\begin{equation}\nonumber
    \forall f \in l^1(G), \quad \|\pi^{ext}(f)\|_{op}\leq \|\pi_{reg}^{ext}(f)\|_{op}.
\end{equation}
Moreover, by \cite{haagerup1978example}, we have:
\begin{equation}\nonumber
    \|\pi_{reg}^{ext}(f)\|_{op} \leq \sum_{k=0}^{\infty}(k+1)\|f1_{S_k}\|_2.
\end{equation}
Consider the sequence of $l^2(G)$ functions $\{f_k\}_k$ defined as
\begin{equation}\nonumber
    f_k(g):=1_{S_k}(g)\overline{\langle\pi(g)\psi_1, \psi_2\rangle} \text{ where } \psi_1,\psi_2 \in \mathcal{H}.
\end{equation}
One has, fixing arbitrary unitary $\psi_1,\psi_2 \in \mathcal{H}$:
\begin{align}\nonumber
    0 &< M_{\psi_1,\psi_2}(S_k) =\langle \pi^{ext}(f_k)\psi_1,\psi_2 \rangle \\ \nonumber
    &\leq \|\pi^{ext}(f_k)\|_{op}\|\psi_1\|_2\|\psi_2\|_2 \leq \|\pi_{reg}^{ext}(f_k)\|_{op} \\ \nonumber
    &\leq (k+1)\|f_k\|_2.
\end{align}
and since $\|f_k\|_2 = M_{\psi_1,\psi_2}(S_k)^{\frac{1}{2}}$, we obtain
\begin{equation}\nonumber
    M_{\psi_1,\psi_2}(S_k) \leq (k+1)^2.
\end{equation}
So, by equation (\ref{average Chandra on sphere}),
\begin{equation}\nonumber
    \frac{M_{\psi_1,\psi_2}(S_k)}{M_{1_{\Omega},1_{\Omega}}(S_k)}\leq \frac{q}{q+1}u_k^2 \quad \left(\text{ for } u_k:=\frac{1+k}{1+\frac{q-1}{q+1}k}\right).
\end{equation}
But one can easily check that $(u_k)_{k \in \N}$ is bounded. So there is a $C \in \R^+$ (which does not depend on $k$) such that:
\begin{equation}\label{bound coeff sphere}
    \forall \psi_1, \psi_2 \in \mathcal{H} , \quad M_{\psi_1,\psi_2}(S_k)\leq CM_{1_{\Omega},1_{\Omega}}(S_k).
\end{equation}
Hence,
\begin{equation}\nonumber
    M_{\psi_1,\psi_2}(B_n)=\sum_{k=0}^nM_{\psi_1,\psi_2}(S_k)\leq \sum_{k=0}^nCM_{1_{\Omega},1_{\Omega}}(S_k)=CM_{1_{\Omega},1_{\Omega}}(B_n).
\end{equation}
\end{proof}

\begin{lemma}
    $\pi$ satisfies the second condition ($\ref{cTemp 2}$).
\end{lemma}
\begin{proof}
    Let $k:=|g|+|h|$\newline
    One can easily show that:
    \begin{equation}\nonumber
        B_n \Delta h^{-1}B_ng \subset B_{n+k}-B_{n-k}
    \end{equation}
    and 
    \begin{align}\nonumber
       & M_{\psi_1,\psi_2}(B_n\Delta h^{-1}B_ng)\leq
        M_{\psi_1,\psi_2}(B_{n+k}-B_{n-k})=\sum_{j=n-k+1}^{n+k} M_{\psi_1,\psi_2}(S^j) \\ \nonumber
        &\overset{\text{ (by \ref{bound coeff sphere})}}{\leq} C\sum_{j=n-k+1}^{n+k} M_{1_{\Omega},1_{\Omega}}(S^j) 
        =C\big(M_{1_{\Omega},1_{\Omega}}(B_{n+k})-M_{1_{\Omega},1_{\Omega}}(B_{n-k})\big)
    \end{align}
    
    Recalling the computation done in (\ref{average Chandra on ball}) which gives:
    \begin{equation}\nonumber
        M_{1_{\Omega},1_{\Omega}}(B_{m})\underset{m\to \infty}{\sim}\frac{m^3}{K}
    \end{equation}
    So
    \begin{align} \nonumber
        &\frac{M_{\psi_1,\psi_2}(B_n\Delta h^{-1}B_ng)}{M_{1_{\Omega},1_{\Omega}}(B_n)}\\ \nonumber
        &\leq C \frac{M_{1_{\Omega},1_{\Omega}}(B_{n+k})-M_{1_{\Omega},1_{\Omega}}(B_{n-k})}{M_{1_{\Omega},1_{\Omega}}(B_n)}\underset{n \to \infty}{\sim}C\frac{(n+k)^3-(n-k)^3}{n^3}.
    \end{align}
    The left hand side converging to zero since the degree 3 coefficient of the numerator vanishes.
\end{proof}

\section{Asymptotic Schur's orthogonality relations for $\pi$.}\label{Asymptotic Schur's orthogonality relations}
Now that we have the conditions ($\ref{cTemp 1}$) and ($\ref{cTemp 2}$) for our representation, we can detail the proof of proposition 2.3 found in \cite{kazhdan2022tempered}. Namely, in our case, for all $\psi_1,\psi_2,\psi_3$ and $\psi_4 \in \mathcal{H}$:
\begin{equation} \label{asymptotic Schur Eq}
    \underset{n \to \infty}{\lim}\frac{\int_{B_n}\langle\pi(g)\psi_1,\psi_2\rangle\overline{\langle\pi(g)\psi_3,\psi_4\rangle}d\mu(g)}{M_{1_{\Omega},1_{\Omega}}(B_n)} = \langle\psi_1,\psi_3\rangle\overline{\langle\psi_2,\psi_4\rangle}.
\end{equation}
We denote by $\overline{\mathcal{H}}$ the conjugate of our vector space $\mathcal{H}$.
This allows us to see any sesquilinear form of $\mathcal{H}$ (like $\langle\pi(g)\cdot,\cdot\rangle, \forall g \in G$) to be a bilinear one on $\mathcal{H}\times \overline{\mathcal{H}}$. 
\begin{lemma}\label{bilinear are multiple}
    Let $B : \mathcal{H}\times \overline{\mathcal{H}} \to \C$ be a bounded bilinear form such that\\
     $B(\pi(g)\psi_1,\pi(g)\psi_3)=B(\psi_1,\psi_3) \quad \forall g \in G$ and $\psi_1,\psi_3 \in \mathcal{H}$.
     Then 
     \begin{equation}\nonumber
         B \in \C\langle\cdot,\cdot\rangle.
     \end{equation}
     In other words, there is a constant $\lambda$ in $\C$ such that $B=\lambda\langle\cdot,\cdot\rangle$.
\end{lemma}
\begin{proof}
    For all $\psi_1$ in $\mathcal{H}$, Riesz lemma gives us an element $T(\psi_1)$ such that $B(\psi_1,\cdot) = \langle T(\psi_1), \cdot \rangle$. \newline
    This defines a map $T:\mathcal{H}\to\C$ which is linear and bounded. Indeed:
    \begin{align}\nonumber
        &\|\underbrace{T(\lambda \psi_1 + \Psi_1) -\lambda T(\psi_1) -T(\Psi_1)}_{=: \psi}\|^2\\ \nonumber
        & =B(\lambda \psi_1 + \Psi_1,\psi) -\lambda B(\psi_1,\psi) -B(\Psi_1,\psi) = 0.
    \end{align}
    and,
    \begin{equation} \nonumber
        \|T(\psi_1)\| \leq \|B\|\|\psi_1\|,
    \end{equation}
    Moreover, $T$ is an intertwining operator since:
    \begin{equation}\nonumber
        \|\underbrace{(T\circ \pi(g)-\pi(g)\circ T)\psi_1}_{=:\psi}\|^2
        = B(\pi(g)\psi_1,\psi) -B(\psi_1,\pi(g^{-1})\psi) = 0.
    \end{equation}
    Hence, by irreducibility of $\pi$ and the application of Schur's Lemma mentioned in (\ref{intertwiners of pi}) implies that $T \in \C Id_{\mathcal{H}}$ and $B=\langle T(\cdot),\cdot\rangle \in \C \langle\cdot,\cdot\rangle$.
\end{proof}
\begin{rema} \label{bilinear are multiple reverse}
    If $D:\overline{\mathcal{H}}\times \mathcal{H} \to \C$ is a bounded bilinear form such that\\
     $\forall g \in G$ and $\psi_2,\psi_4 \in \mathcal{H}, \quad D(\pi(g)\psi_2,\pi(g)\psi_4)=B(\psi_2,\psi_4)$.\\
     Then, composing it with the flip operator $F$ which swap the coordinates,
     we obtain that $B:= D \circ F$ satisfies the conditions of the previous lemma (\ref{bilinear are multiple}). So
     $D \circ F \in \C \langle\cdot,\cdot\rangle$, that is to say $D \in \C \overline{\langle\cdot,\cdot\rangle}$.
\end{rema}
One last simple lemma ( about convergence in $\C$) before proving the equation (\ref{asymptotic Schur Eq}) mentioned at the beginning of the section:
\begin{lemma} \label{converg subSeq}
    Let $(u_n)_{n \in \N} \in \C^{\N}$ and $l \in \C$. Then
    \begin{equation}\nonumber
        \underset{n \to \infty}{\lim}u_n = l \text{ is equivalent to the following condition:}
    \end{equation}
     \begin{align}\nonumber 
     &\text{For all subsequences } (u_{\alpha(n)})_{n \in \N}, \text{ there is a further subsequence } (u_{\alpha \circ \gamma(n)})_{n \in \N}\\ \nonumber
     &\text{ such that } \underset{n \to \infty}{\lim}u_{\alpha \circ \gamma(n)} = l.
     \end{align}
\end{lemma}

\begin{theo}
    Given any $\psi_1,\psi_2,\psi_3$ and $\psi_4 \in \mathcal{H}$:
\begin{equation}  \nonumber\label{asymptotic Schur Eq THM}
    \underset{n \to \infty}{\lim}\frac{\int_{B_n}\langle\pi(g)\psi_1,\psi_2\rangle\overline{\langle\pi(g)\psi_3,\psi_4\rangle}d\mu(g)}{M_{1_{\Omega},1_{\Omega}}(B_n)} = \langle\psi_1,\psi_3\rangle\overline{\langle\psi_2,\psi_4\rangle} .
\end{equation}
\end{theo}
\begin{proof}
   Define $\forall n \in \N$:
    \begin{equation}\nonumber
        S_n: \mathcal{H} \times \overline{\mathcal{H}} \to L^2(G,\mu), \text{ by}
    \end{equation}
    \begin{equation}\nonumber
        S_n(\psi_1,\psi_2)(g):= 1_{B_n}\frac{\langle \pi(g)\psi_1, \psi_2\rangle}{M_{1_{\Omega},1_{\Omega}}(B_n)^{\frac{1}{2}}}.
    \end{equation}
    Then $S_n$ is clearly bilinear and, since $\pi$ is c-tempered, we have by the first condition of c-temperedness (\ref{cTemp 1}),
    \begin{equation}\nonumber
        \underset{n \to \infty}{\limsup}\int_{\omega}|S_n(\psi_1,\psi_2)(g)|^2d\mu(g) = \underset{n \to \infty}{\limsup}\frac{M_{\psi_1,\psi_2}(B_n)}{M_{1_{\Omega},1_{\Omega}}(B_n)}<\infty.
    \end{equation}
    This shows that $(S_n(\psi_1,\psi_2))_{n \in \N}$ bounded in $L^2(G)$ and,
    by the Banach-Steinhaus theorem, 
    \begin{equation} \label{maj Sn}
        \exists C \in \R^+ \text{ such that }\|S_n\|^2 \leq C.
    \end{equation}
    Now, $\forall n \in \N$, one can define the quadrilinear form
    \begin{align}\nonumber
        &Q_n:\mathcal{H} \times \overline{\mathcal{H}} \times \overline{\mathcal{H}}\times \mathcal{H} \to \C \text{ as:}\\ \nonumber
        &Q_n(\psi_1,\psi_2,\psi_3,\psi_4) := \langle S_n(\psi_1,\psi_2),S_n(\psi_3,\psi_4)\rangle.
    \end{align}
    $(Q_n)_{n \in \N}$ is also uniformly bounded, since, using Cauchy-Schwartz inequality and equation (\ref{maj Sn}), we see, for unitary $\psi_1, \psi_2, \psi_3$ and $\psi_4$:
    \begin{equation}\nonumber
        \|Q_n(\psi_1,\psi_2,\psi_3,\psi_4)\|\leq \|S_n(\psi_1,\psi_2)\|\|S_n(\psi_3,\psi_4)\|\leq C.
    \end{equation}
    Next we will show that, $\forall g,h \in G$, we have:
    \begin{equation}\label{Qn asymptotic  pi X pi invariant}
        \underset{n \to \infty}{\lim}\bigg(Q_n\big(\pi(g)\psi_1,\pi(h)\psi_2,\pi(g)\psi_3,\pi(h)\psi_4\big)-Q_n\big(\psi_1,\psi_2,\psi_3,\psi_4\big)\bigg)=0.
    \end{equation}
    Indeed, one can show that:
    \begin{align} \nonumber
        &\left|Q_n\big(\pi(g)\psi_1,\pi(h)\psi_2,\pi(g)\psi_3,\pi(h)\psi_4\big)-Q_n\big(\psi_1,\psi_2,\psi_3,\psi_4\big)\right|\\ \nonumber
        &\leq \bigg(\frac{M_{\psi_1,\psi_2}(B_n\Delta h^{-1}B_ng)}{M_{1_{\Omega},1_{\Omega}}(B_n)}\bigg)^\frac{1}{2}\bigg(\frac{M_{\psi_3,\psi_4}(B_n\Delta h^{-1}B_ng)}{M_{1_{\Omega},1_{\Omega}}(B_n)}\bigg)^\frac{1}{2}
    \end{align}
     which tends to $0$ when $n$ goes to infinity since $\pi$  is c-tempered (see condition 2 of c-temperedness (\ref{cTemp 2})).\newline
    Now, One can consider any subsequence $(Q_{\alpha(n)})_{n \in \N}$ and an application of the Banach-Alaoglu theorem in $\overline{B_{\|\|}(0,C)}$ gives us a further subsequence $(Q_{\alpha \circ \gamma(n)})_{n \in \N}$ which converges point-wise to some $Q$ in $\overline{B_{\|\|}(0,C)}$:\newline
    \begin{equation}\label{Qn has a converg subs}
       \forall \psi_1,\psi_2,\psi_3,\psi_4, \quad Q\big(\psi_1,\psi_2,\psi_3,\psi_4\big) =\underset{n \to \infty}{\lim}Q_{\alpha \circ \gamma(n)}\big(\psi_1,\psi_2,\psi_3,\psi_4\big) \text{ in } \C.
    \end{equation}
     Now, by (\ref{Qn asymptotic  pi X pi invariant}), we have:
    \begin{align}\nonumber
        &Q\big(\pi(g)\psi_1,\pi(h)\psi_2,\pi(g)\psi_3,\pi(h)\psi_4\big)
        =Q\big(\psi_1,\psi_2,\psi_3,\psi_4\big)\\ \nonumber
        &\text{and } Q \text{ is also quadrilinear.}
    \end{align}
In particular, if we fix $(\psi_2, \psi_4) \in \overline{\mathcal{H}} \times \mathcal{H}$ and consider the map
\begin{equation}\nonumber
    Q(\cdot,\psi_2,\cdot,\psi_4) \text{ (which satisfies the conditions of lemma \ref{bilinear are multiple})}.
\end{equation}
We know that, $\exists \lambda_{\psi_2,\psi_4} \in \C$ such that: 
\begin{equation}\nonumber
    Q(\cdot,\psi_2,\cdot,\psi_4) = \lambda_{\psi_2,\psi_4}\langle\cdot,\cdot\rangle.
\end{equation}
Similarly, fixing $(\psi_1, \psi_3) \in \mathcal{H} \times \overline{\mathcal{H}}$ and by the remark \ref{bilinear are multiple reverse} following the mentioned lemma, we have $\exists \lambda^{\psi_1,\psi_3} \in \C$ such that:
\begin{equation}\nonumber
    Q(\psi_1,\cdot,\psi_3,\cdot) = \lambda^{\psi_1,\psi_3}\overline{\langle\cdot,\cdot\rangle}.
\end{equation}
Next, using the definition of these coefficient one calculate $\lambda_{1_{\Omega},1_{\Omega}}=1$, $\lambda^{\psi_1,\psi_3}=\langle\psi_1,\psi_3\rangle$ and
$Q(\psi_1,\psi_2,\psi_3,\psi_4)=\lambda^{\psi_1,\psi_3}\overline{\langle\psi_2,\psi_4\rangle}=\langle\psi_1,\psi_3\rangle\overline{\langle\psi_2,\psi_4\rangle}$.\\
To finish the proof,  Let $ \psi_1,\psi_2,\psi_3,\psi_4 \in \mathcal{H}$ and define:
    \begin{equation}\nonumber
        u_n := \frac{\int_{B_n}\langle\pi(g)\psi_1,\psi_2\rangle\overline{\langle\pi(g)\psi_3,\psi_4\rangle}d\mu(g)}{M_{1_{\Omega},1_{\Omega}}(B_n)}
    \end{equation}
    and consider any subsequence $(u_{\alpha(n)})_{n \in \N}$.\newline
    Since   $u_{\alpha(n)} = Q_{\alpha(n)}(\psi_1,\psi_2,\psi_3,\psi_4)$, has further subsequence converging to $\langle\psi_1,\psi_3\rangle\overline{\langle\psi_2,\psi_4\rangle}$, we have by the lemma (\ref{converg subSeq}), $\underset{n \to \infty}{\lim}u_n = \langle\psi_1,\psi_3\rangle\overline{\langle\psi_2,\psi_4\rangle}$.
\end{proof}
\begin{rema}
    A unitary representation which satisfies these orthogonality relations has to be irreducible.
\end{rema}

So, using the equation result (\ref{asymptotic Schur Eq}) of our last theorem, and the computation of $M_{1_{\Omega},1_{\Omega}}(B_n)$ we did in (\ref{average Chandra on ball}), we obtain the concrete result mentioned at the beginning of this paper:
\begin{equation} \nonumber
      \underset{n \to \infty}{\lim}\frac{1}{n^{3}}\int_{B_n}\langle\pi(g)\psi_1,\psi_2\rangle\overline{\langle\pi(g)\psi_3,\psi_4\rangle}d\mu(g) = \frac{3q(q+1)}{(q-1)^2}\langle\psi_1,\psi_3\rangle\overline{\langle\psi_2,\psi_4\rangle}.
\end{equation}

\bibliography{main}

\end{document}